\documentclass[11pt]{amsart}
\usepackage[usenames]{color}
\usepackage{amsmath,amssymb,amsthm,amsfonts}
\usepackage[colorlinks=true,
linkcolor=webgreen,
filecolor=webbrown,
citecolor=webgreen]{hyperref}
\usepackage{graphicx}
\usepackage{tikz}
\usetikzlibrary{math,calc}
\usepackage{url}
\usepackage{relsize}

\definecolor{webgreen}{rgb}{0,.5,0}
\definecolor{webbrown}{rgb}{.6,0,0}

\newtheorem{thm}{Theorem}[section]
\newtheorem{prop}[thm]{Proposition}

\theoremstyle{remark}

\newtheorem*{rmk*}{Remark}

\newtheorem*{ex*}{Example}

\title{Sizes of Simultaneous Core Partitions}

\author{Chaim Even-Zohar}
\address{Chaim Even-Zohar, The Alan Turing Institute, London, NW1 2DB, UK}
\email{chaim@ucdavis.edu}

\begin{document}
\maketitle

\begin{abstract}
\vspace*{-3em}
There is a well-studied correspondence by Jaclyn Anderson between partitions that avoid hooks of length $s$ or $t$ and certain binary strings of length $s+t$. Using this map, we prove that the total size of a random partition of this kind converges in law to Watson's~$U^2$ distribution, as conjectured by Doron Zeilberger.
\vspace*{1em}
\end{abstract}


\newcommand{\boxes}[1]{\raisebox{-2pt}{\tikz{
\foreach \i/\k in {#1} \foreach \j in {1,...,\i} 
\draw (0.12*\j, -0.12*\k) rectangle (0.12*\j+0.12, -0.12*\k+0.12);}}}

A \emph{partition} is a finite set of square boxes stacked in the upper left corner, as in Figure~\ref{hook}. Every box in a partition has a \emph{hook}, the set of boxes directly to its right or below it. A~partition is called $p$-core if no hook has exactly $p$ boxes. Such partitions arise in the context of $p$-modular representations of the symmetric group, and go back to Nakayama~\cite{james1981representation, nakayama1940some}.

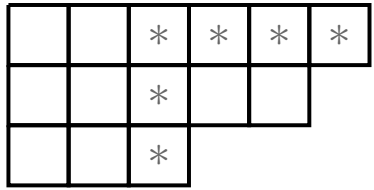
\begin{figure}[tb]
\begin{tikzpicture}[scale=0.8]
\foreach \x/\y/\z in {1/1/,1/2/,1/3/,2/1/,2/2/,2/3/,3/1/$\ast$,3/2/$\ast$,3/3/$\ast$,4/1/$\ast$,4/2/,5/1/$\ast$,5/2/,6/1/$\ast$}
{\draw[black,line width=1.5pt]
(\x,-\y) -- (\x+1,-\y) -- (\x+1,-\y-1) -- (\x,-\y-1) -- (\x,-\y); \node[color=black!60] at(\x+0.5,-\y-0.5){\LARGE \z};}
\pgfresetboundingbox \clip (0.5,-4.5) rectangle (7.5,-0.5);
\end{tikzpicture}
\caption{An example of a partition with 14 boxes. It is denoted by the row sizes: \textrm{6\,5\,3}. The hook that corresponds to the third box in the first row is made of the 6 boxes marked by~$\ast$. Therefore this partition is \emph{not} a 6-core.}
\label{hook}
\end{figure}

There are infinitely many $p$-core partitions for every~$p \geq 2$. However, if $s$ and $t$ are coprime then only finitely many partitions are simultaneously $s$-core and $t$-core, or for short \emph{$(s,t)$-core}. See \cite{gorsky2014compactified, gorsky2014affine, armstrong2014results, thiel2016anderson, thiel2017strange} for algebraic structures related to simultaneous core partitions.

A~wonderful bijection by Anderson relates $(s,t)$-core partitions to Dyck paths in an $s \times t$ rectangle, or equivalently to $s/t$-ballot words~\cite{anderson2002partitions}. An \emph{$s/t$-ballot word} is an $(s+t)$-long string of $S$ and~$T$, such that the proportion of occurrences $\#S/\#T$ in any prefix of the string is at least its overall proportion~$s/t$. The so-called \emph{rational Catalan number}~$\tfrac{1}{s+t}\tbinom{s+t}{s}$ gives the count of such words, and hence, the number of $(s,t)$ core partitions.

A conjecture by Armstrong~\cite[\S2]{armstrong2014results}, that $(s,t)$ core partitions have $\tfrac{1}{24}(s+t+1)(s-1)(t-1)$ boxes on average, has been proven by several methods in several works \cite{stanley2015catalan, aggarwal2015armstrong, johnson2018lattice, wang2016simultaneous, ekhad2015explicit}. It has also been shown that the variance is $\tfrac{1}{1440}(s+t+1)(s+t)s(s-1)t(t-1)$, and higher moments have been similarly computed by efficient algorithms \cite{thiel2017strange, ekhad2015explicit}. Based on their leading terms and an online search, a precise limit distribution for large $s$ and~$t$ has been conjectured by Zeilberger \cite{ekhad2015explicit, zaleski2017explicit-13}. The following theorem meets this challenge.

\begin{thm} \label{stcore}
Let $X_{st}$ be the total size of a uniformly random $(s,t)$ core partition for coprime $s$ and $t$. Then
$$ \frac{X_{st}}{st(s+t)/2} \;\;\;\xrightarrow[\;\;s \to \infty, \; t \to \infty\;\;]{D}\;\;\; U^2 \;\sim\; \sum\limits_{k=1}^{\infty}\frac{Z_k^2+\tilde{Z}_k^2}{4\pi^2k^2} $$
where all $Z_k$ and $\tilde{Z}_k$ are mutually independent standard normal random variables.
\end{thm}

We note that the limiting expectation is $E[U^2] = 1/12$ in accordance with Armstrong's conjecture, and the variance is $V[U^2] = 1/360$. The tail probability $P[U^2 > t]$ is given by  $2\sum_{m=1}^{\infty}(-1)^{m-1}\exp(-2m^2\pi^2t)$. This asymptotic distribution originates in Watson's~$U^2$ test for goodness of fit on a circle~\cite{watson1961goodness}, and its two-sample variant~\cite{watson1962goodness}. We provide additional background on these statistics later on, and show that especially the latter is closely related to $X_{st}$, even for finite $s$ and~$t$. 

\begin{rmk*}
Theorem~\ref{stcore} addresses the first challenge in~\cite{ekhad2015explicit}, without the assumption that $s-t$ is fixed. It applies, for example, to $(s,2s+1)$ core partitions, or $(s,s^s+1)$ core partitions.
\end{rmk*}

The main ingredients in the proof of Theorem~\ref{stcore} are a new formula for the number of boxes in an $(s,t)$ core partition, and a similar formula due to Persson for Watson's statistic~\cite{persson1979new}. The new formula, given in Proposition~\ref{stst} below, relates the partition's size to the occurrences of $STST$ and $TSTS$ in the corresponding $s/t$ ballot word. 

\begin{rmk*}
Proposition~\ref{stst} also provides a different approach to Johnson's result~\cite{johnson2018lattice}, that all the moments of the partition's size are polynomials in $s$ and $t$, arguably in less than two pages. This is the second challenge in~\cite{ekhad2015explicit}. 
\end{rmk*}

Since these challenges were posed, much attention has been paid to restricted families of core partitions. These include partitions into \emph{distinct} parts, that are $(s,s+1)$ core \cite{amdeberhan2015theorems, straub2016core, zaleski2017explicit-ss1, paramonov2018cores, johnson2018simultaneous, xiong2018core}, or $(s,s+2)$ core \cite{yan20172k, zaleski2017explicit-13, baek2018bijective, paramonov2018cores}, or
$(s,ds \pm 1)$ core \cite{aggarwal2015armstrong,straub2016core, nath2017abaci, zaleski2017explicit-d, xiong2018largest, xiong2019polynomiality}, and similarly partitions into \emph{odd} parts \cite{zaleski2017intriguing, johnson2018simultaneous}. Also \emph{self-conjugate} $(s,t)$ core partitions have been studied \cite{ford2009self, chen2016average, wang2016simultaneous, wang2018moments}, and further restricted families that avoid more than two hook lengths, such as $(s,s+1,s+2)$ core partitions \cite{amdeberhan2015multi, aggarwal2015does,  amdeberhan2015theorems, yang2015enumeration, xiong2016largest, baek2019johnson}. 

The enumeration of restricted partitions has been established in various cases, as well as their maximum and average size, and some higher moments. Remarkably, the size distribution of $(s,s+1)$ core partitions into distinct parts has been shown to be asymptotically normal~\cite{komlos2018asymptotic}. While Theorem~\ref{stcore} settles the asymptotics of the fundamental case of general $(s,t)$-core partitions, the formula in Proposition~\ref{stst} is applicable also to restricted cases, such as those mentioned above.

\subsection*{Plan}
First, we describe the correspondence by Anderson between simultaneous core partitions and ballot words. Next, we state and prove the new formula for the size of a partition. Then, we give some statistical background on Watson's $U^2$ distribution. Finally, we combine all these ingredients and deduce Theorem~\ref{stcore}.

\medskip

\subsection*{Anderson's Bijection}
\cite{anderson2002partitions}
An \emph{$s/t$-ballot word} is a binary string $w$ of length $s+t$ over the alphabet~$\{S,T\}$, such that the numbers of appearances of the two letters are $\#S(w)=s$ and $\#T(w)=t$, and for every prefix $p$ of $w$ their appearance ratio satisfies $\#S(p)/\#T(p) \geq s/t$. For example $STSTT$ is a $2/3$-ballot word. We note that this is equivalent to a generalized Dyck path, which is a staircase walk from $(0,0)$ to $(s,t)$ in a rectangular $s \times t$ grid that lies above the diagonal connecting these two corners.

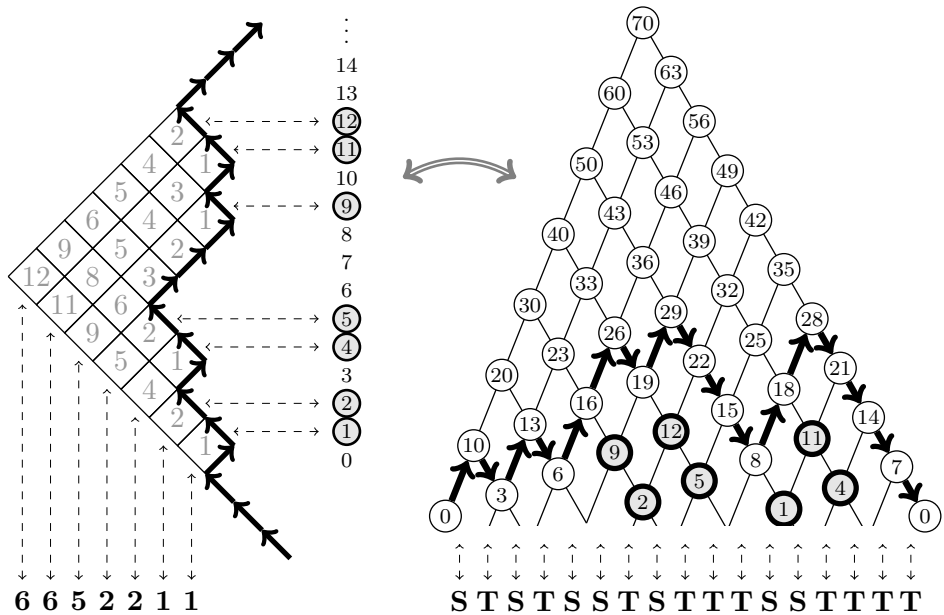
\begin{figure}[b]
\begin{tikzpicture}[line width=0.75pt,black,scale=0.375]
\foreach \y in {1,2,4,5,9,11,12}
\node[circle, draw, line width = 1pt, fill=black!10, inner sep=0pt, minimum width=10pt] at(-1,\y)(0){\scriptsize \y};
\foreach \y in {0,3,6,7,8,10,13,14}
\node at(-1,\y){\scriptsize \y};
\node at(-1,15.5){\scriptsize $\vdots$};
\foreach \x/\y/\z in {7/1/1, 8/2/2, 9/3/4, 10/4/5, 11/5/9, 12/6/11, 13/7/12, 8/4/1, 9/5/2, 10/6/6, 11/7/8, 12/8/9, 9/7/3, 10/8/5, 11/9/6, 8/8/2, 9/9/4, 10/10/5, 7/9/1, 8/10/3, 9/11/4, 7/11/1, 8/12/2}
{\draw[black,line width=0.5pt]
(-\x,\y-0.5) -- (-\x+1,\y+0.5) -- (-\x+2,\y-0.5) -- (-\x+1,\y-1.5) -- (-\x,\y-0.5); \node[color=black!40] at(-\x+1,\y-0.5){\z};}
\foreach \y/\x in {1/5, 2/6, 4/6, 5/7, 9/5, 11/5, 12/6}
{\draw[->,line width=2pt] (-\x,\y-0.5)--(-\x-1,\y+0.5); \draw[<->,dashed, line width=0.25] (-2,\y)--(-\x,\y);}
\foreach \y/\x in {-3/3, -2/4, -1/5}
{\draw[->,line width=2pt] (-\x,\y-0.5)--(-\x-1,\y+0.5);}
\foreach \y/\x in {0/6, 3/7, 6/8, 7/7, 8/6, 10/6, 13/7, 14/6, 15/5}
\draw[->,line width=2pt] (-\x,\y-0.5)--(-\x+1,\y+0.5);
\foreach \x in {-14,-13,-12, -11,-10,-9,-8}\draw[<->,dashed, line width=0.25] (\x+1.5,-4.25)--(\x+1.5,-8.5-\x);
\foreach \x/\y in {-14/6, -13/6, -12/5, -11/2, -10/2, -9/1, -8/1} \node at(\x+1.5,-5){\textbf\y};
\def\shift{19.5}
\foreach \y/\x in {70/10, 63/9, 60/11, 56/8, 53/10, 50/12, 49/7, 46/9, 43/11, 42/6, 40/13, 39/8, 36/10, 35/5, 33/12, 32/7, 30/14, 29/9, 28/4, 26/11, 25/6, 23/13, 22/8, 21/3, 20/15, 19/10, 18/5, 16/12, 15/7, 14/2, 13/14, 10/16, 8/6, 7/1, 6/13, 3/15} 
{ \draw[-,line width=0.5pt] (\shift-\x+1,\y/4-7/4-2) -- (\shift-\x,\y/4-2) -- (\shift-\x-1,\y/4-10/4-2);
\node[circle, line width=0.5pt, draw, fill=black!0, inner sep=0pt, minimum width=12pt] at(\shift-\x,\y/4-2)(\y){\scriptsize \y};}
\foreach \y/\x in {12/9, 11/4, 9/11, 5/8, 4/3, 2/10, 1/5} 
{ \draw[-,line width=0.5pt] (\shift-\x+1,\y/4-7/4-2) -- (\shift-\x,\y/4-2) -- (\shift-\x-1,\y/4-10/4-2);
\node[circle, line width=2pt, draw, fill=black!10, inner sep=0pt, minimum width=12pt] at(\shift-\x,\y/4-2)(\y){\scriptsize \y};}
\fill [fill=white] (0,-2.35) rectangle (\shift-1,-5.6);
\foreach \y/\x in {0/0, 0/17} 
{ \node[circle, line width=0.5pt, draw, fill=black!0, inner sep=0pt, minimum width=12pt] at(\shift-\x,\y/4-2)(\x){\scriptsize \y};}
\foreach \x/\y in {0/7, 7/14, 14/21, 21/28, 28/18, 18/8, 8/15, 15/22, 22/29, 29/19, 19/26, 26/16, 16/6, 6/13, 13/3, 3/10, 10/17}
\draw[<-, line width=2.5] (\x) -- (\y);
\foreach \x/\y in {0/T, 1/T, 2/T, 3/T, 4/S, 5/S, 6/T, 7/T, 8/T, 9/S, 10/T, 11/S, 12/S, 13/T, 14/S, 15/T, 16/S} \node at(\shift-\x-0.5,-5){\textbf\y};
\foreach \x in {0,1,2,3, 4,5,6, 7,8,9, 10,11,12, 13,14,15, 16}\draw[<->,dashed, line width=0.25] (\shift-\x-0.5,-4.25)--(\shift-\x-0.5,-3);
\draw[<->,bend right,color=black!50,line width=1pt,double=white] (1,10) to[out=30,in=150] (5,10);
\end{tikzpicture}
\caption{A demonstration of Anderson's two-step bijection between a $(7,10)$ core partition and a $7/10$ ballot word. Here the downset $A = \{1,2,4,5,9,11,12\}$.}
\label{bijection}
\end{figure}

The first step of the bijection defines a set of natural numbers. An~$(s,t)$ core partition can be reconstructed from the set $A \subset \mathbb{N}$ of hook sizes of its first-column boxes. Indeed, walk along the rim of the partition, rotated by 45 degrees as on the left hand side of Figure~\ref{bijection}. If the first right turn is taken at step~$0$, then $A$~is exactly the set of up-left rim steps.

The avoidance of $s$-hooks and $t$-hooks amounts to the requirement that $a \in A$ if $a+s \in A$ or $a+t \in A$. In other words, $A$ is a downset with respect to the partial order of $\mathbb{N}$ generated by $a \prec a+t$ and $a \prec a+s$, such that~$0 \not\in A$. A~\emph{downset} in a poset is a downward-closed subset, also known as an ideal. Since $s$ and~$t$ are coprime and $A \subseteq \mathbb{N} \setminus (s\mathbb{N}+t\mathbb{N})$, the number of $(s,t)$ core partitions is finite. The right hand side of Figure~\ref{bijection} shows the Hasse diagram of elements below $st$ in this partial order, for $(s,t)=(7,10)$. Clearly, the Hasse diagram would have such a triangular shape with square-grid texture for general coprime $s$ and~$t$.

The elements immediately above $A$ form a closed path through~$0$ in the Hasse diagram, with $s$ steps up and $t$ steps down. We record this path by a word $w$ with $s$ times $S$ and $t$ times $T$. Every such word uniquely describes a downset $A \not\ni 0$, as long as $\#S(p)\cdot t-\#T(p) \cdot s$ is non-negative for every prefix $p$ of $w$. That is, $w$ may be any $s/t$ ballot word.

\medskip

\subsection*{The Size Formula}
Now we are ready to state the relation between the size of an $(s,t)$ core partition and the corresponding $s/t$ ballot word. 

Here and below, $\#u(w)$ denotes the number of occurrences of the word~$u$ as a subsequence in the word~$w$, not necessarily as consecutive characters. For example $\#TS(STSS)=2$. It may be noted that in~\cite{even2020spectral} we provide a systematic analysis of this kind of subword statistics. 

\begin{prop}\label{stst}
The size of the $(s,t)$ core partition that corresponds, via Anderson's bijection, to the $s/t$ ballot word $w$ is
$$ \frac{(s^2-1)(t^2-1)}{24} \;-\; \frac{\#STST(w) + \#TSTS(w)}{2} $$
\end{prop}

\begin{ex*}
We demonstrate the size formula using all $(2,5)$-core partitions. There are $\tfrac{1}{2+5}\tbinom{2+5}{2} = 3$ such partitions, listed in the table below. For each partition, we first compute the downset of hook sizes in the leftmost column, and then use the Hasse diagram to find the corresponding  2/5-ballot word. Then we count $STST$ and $TSTS$ and verify the size formula. The constant term is ${(2^2-1)(5^2-1)}/{24}=3$ in this case.
\bigskip

\noindent
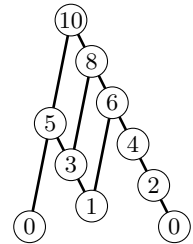
\begin{tabular}{ccccc}
\textbf{Partition} & \textbf{Downset} & \textbf{Word} & {\footnotesize \!$\#STST{+}\#TSTS$} & \textbf{Size} \\[0.2em] \hline\hline \\[-0.8em] 
\boxes{1/2,2/1} & $\{1,3\}$ & {\footnotesize $SSTTTTT$} & 0 & 3 \\[0.4em]
\boxes{1/1} & $\{1\}$ & {\footnotesize $STSTTTT$} & 4 & 1 \\[0.4em]
. & $\varnothing$ & {\footnotesize $STTSTTT$} & 6 & 0 
\end{tabular}
\!\!
\raisebox{-0.5\height}{
\begin{tikzpicture}[scale=0.275]
\draw[black,line width=1pt] (0,0) -- (2,10) -- (7,0) (1,5) -- (3,1) -- (4,6) (2,3) -- (3,8);
\foreach \s/\t [evaluate={\h=int(5*\s-2*\t);}] in {0/0,1/0,2/0,1/1,2/1,1/2,2/2,2/3,2/4,2/5}
\node[circle,draw=black,fill=white,inner sep=0pt, minimum size=12pt] at(\s+\t,\h){{\footnotesize \h}};
\end{tikzpicture}}
\end{ex*}

\bigskip

\begin{proof}
Any $s/t$ ballot word may be transformed to the word $SS...STT...T$ by a sequence of adjacent transpositions of the form $TS \to ST$. All the intermediate words are also $s/t$ ballot. Therefore, we proceed by induction on such swaps.

Consider an $s/t$ ballot word $w = pTSq$ with any prefix $p$ and suffix $q$, and let $w' = pSTq$. For the induction step $w' \to w$, we look on the following difference of pattern counts between the words.
$$ \Delta \;=\; \frac{\#STST(w) + \#TSTS(w)}{2} \;-\; \frac{\#STST(w') + \#TSTS(w')}{2} $$

Occurrences of $STST$ or $TSTS$ that do not involve the two swapped characters are the same in $w$ and~$w'$. Also occurrences with only one of the middle $S$ and~$T$ are the same, letting this character move one position. However, $STST$ and $TSTS$ that use both middle characters do not survive the swap. We divide into six cases according to the other two characters.
$$ 2\Delta \;=\; \#TS(p) + \#TS(q) + \#S(p)\#T(q) - \#ST(p) - \#ST(q) - \#T(p)\#S(q) $$
For example, the product $\#S(p)\#T(q)$ corresponds to all occurrences of $STST$ in $w = pTSq$ that contain the swapped $TS$, and thus do not have corresponding occurrences in $w'=pSTq$.

To simplify, we use $\#TS(w)$ and $\#ST(w)$, expanded in the different cases for their occurrences in $p$ and~$q$.
\begin{align*}
\#TS(w) \;&=\; \#TS(p) + \#TS(q) + \#T(p)\#S(q) + \#T(p) + \#S(q) + 1 \\
\#ST(w) \;&=\; \#ST(p) + \#ST(q) + \#S(p)\#T(q) + \#S(p) + \#T(q)
\end{align*}
These two relations yield
\begin{align*}
2\Delta \;=\;& \;\#TS(w) - \#ST(w)  +2 \#S(p)\#T(q) -2 \#T(p)\#S(q) \\
&\; - \#T(p) + \#T(q) - \#S(q) + \#S(p) - 1
\end{align*}
Here are another three immediate relations.
\begin{align*}
s \;&=\; \#S(w) \;=\; \#S(p) + \#S(q) + 1 \\
t \;&=\; \#T(w) \;=\; \#T(p) + \#T(q) + 1 \\
st \;&=\; \#S(w)\#T(w) \;=\; \#TS(w) + \#ST(w)
\end{align*}
The following expression for $\Delta$ now follows from the previous one.
$$ \Delta \;=\; \#S(p)\cdot t - \left(\#T(p)+1\right)\cdot s  \;+\; \#TS(w) - \tfrac12(s-1)(t-1) $$

This has a meaningful interpretation in terms of the Hasse diagram in Anderson's bijection. Let $A$ be the downset below the path of $w$, and let $A'$ be the downset of~$w'$. Since $w$ and $w'$ differ in one $TS \to ST$ transposition, $A' = A \cup \{a\}$ for some $a \not\in A$. The prefix $pT$ of~$w$ provides a path from $0$ to~$a$ in the diagram. The steps along this path add up to
$$ a \;=\; \#S(p)\cdot (+t) \;+\; \left(\#T(p)+1\right)\cdot (-s) $$ 
We also express the size of $A$ in terms of~$w$. In the case $w=SS...STT...T$, the downset $A$ is $\mathbb{N}\setminus(s\mathbb{N}+t\mathbb{N})$, which are all the elements below the path $(0,t,2t,\dots(s-1)t,st,s(t-1),\dots,2s,s,0)$ in the Hasse diagram of the partial order. There are $\tfrac12(s-1)(t-1)$ such elements, the ``area'' of the triangle as in Figure~\ref{bijection}. For other $w$, we pop one element from $A$ at each swap $ST \to TS$, hence
$$ |A| \;=\; \tfrac12(s-1)(t-1) \;-\; \#TS(w) $$
In conclusion, the reduction $w \to w'$ adds $\Delta = a - |A|$ to the formula in the proposition, where $A$ and $A \cup \{a\}$ are the first-column hook sizes in the partitions corresponding to $w$ and $w'$.  

We show that $\Delta$ is also the size difference between the $(s,t)$-core partitions. Indeed, every box in the partition has a hook that starts at some rim step $b \not\in A$ and ends at some rim step $a \in A$. By counting boxes, the size of the $(s,t)$-core partition is
$$ \#\left\{(a,b)\in \mathbb{N}^2\;|\;a > b,\; a \in A,\; b \not\in A\right\} \;\;=\;\; \sum_{a \in A} a \;-\; \binom{|A|}{2} $$
This quantity increases by $a-|A|$ upon insertion of a new element $a$ to $A$. This completes the induction step between $w$ and $w'$ in the proof of the proposition. 

It is left to determine a global additive shift that only depends on $s$ and~$t$, which should be given by the first term in the proposition. Since for $w = SS...STT...T$ the second term vanishes, this must be the size of the partition that corresponds to $A=\mathbb{N}\setminus(s\mathbb{N}+t\mathbb{N})$. This extreme case was solved by Olsson and Stanton who proved that the number of boxes is~$(s^2-1)(t^2-1)/24$~\cite{olsson2007block}, cf.~\cite{tripathi2009largest, johnson2018lattice}.
\end{proof}

\begin{rmk*}
It is interesting to compare the size formula of Proposition~\ref{stst} with the one used by Johnson~\cite{johnson2018lattice}, who related $(s,t)$-core partitions to integer points in lattice polytopes and Ehrhart theory. 

In short, one can assign to every $(s,t)$-core partition a point in $\mathbb{N}^s$, given by the positions of the letter~$S$ in the corresponding $s/t$-ballot word, in increasing order. The image of this map is the integer points in a certain simplex in~$\mathbb{R}^s$ which depends on~$t$. The word statistics $\#STST$ and $\#TSTS$ are quadratic polynomials in the coordinates of the corresponding lattice point, and therefore, so is the size of the $(s,t)$-core partition. Up to a linear transformation on the coordinates, Johnson derived such a quadratic expression for the partition's size in Lemma~26 of~\cite{johnson2018lattice}, and then used it to prove Armstrong's conjecture.  
\end{rmk*}

\medskip

\subsection*{Watson's Statistics}

Before using the size formula to prove Theorem~\ref{stcore}, we provide additional background on the statistical results to which the problem is reduced.

In statistics, a measure of \emph{goodness of fit} aims to quantify the discrepancy between a theoretical distribution over some space and an empirical distribution which is based on observed values. Also in the \emph{two-sample} setting, a~measure of \emph{similarity} compares two unknown distributions using two respective sets of observations. Such statistical measures are commonly used in \emph{hypothesis testing}, to decide whether to reject the null hypotheses, that the observations fit the specified distribution or that the two samples originate in the same underlying distribution.

The classical \emph{Cram\'er--von Mises criterion} addresses the case of continuous distributions on the real line~\cite{cramer1928composition,mises1931wahrscheinlichkeitsrechnung}. It is defined as the Lebesgue–Stieltjes integral
$$ \omega_n^2 \;=\; \int\limits_{-\infty}^{\infty} \left[F_n(x) - F(x)\right]^2 dF(x) $$
The theoretical distribution is specified by the cumulative distribution function $F(x) = P(X \leq x)$ of the random variable~$X$, and $F_n$~is the empirical distribution $F_n(x) = \#\{i:X_i \leq x\}/n$, of the given sample $X_1,\dots,X_n$.

The analogous two-sample statistic compares the empirical distributions of two real-valued samples $X_1,\dots,X_m$ and $Y_1,\dots,Y_n$ \cite{lehmann1951consistency,rosenblatt1952limit}.
$$ \omega_{mn}^2 \;=\; \int\limits_{-\infty}^{\infty} \left[F_m(x) - G_n(x)\right]^2 dH(x) $$
Here $F_m$ and $G_n$ are the empirical distributions of the two respective given samples $X_1,\dots,X_m$ and $Y_1,\dots,Y_n$, and the mixed empirical distribution of both samples is $H = \tfrac{m}{m+n}F_m + \tfrac{n}{m+n}G_n$.

Watson derived two corresponding measures for samples that are drawn from distributions on a circle, rather than the real line~\cite{watson1961goodness, watson1962goodness}. It is tempting to map the circle to a real interval by cutting it at some point, and then compute $\omega_n^2$ or $\omega_{mn}^2$ as before. Since the resulting statistics depend on the arbitrary cutting point, Watson proposed new statistics that address this issue. Wastson's
$$ U_{n}^2 \;=\; n \int\limits_{-\infty}^{\infty} \left[F_n(x) - F(x)  - \textstyle\int\limits_{-\infty}^{\infty} \left[F_n(r) - F(r)\right] dF(r) \right]^2 dF(x) $$
and
$$ U_{mn}^2 \;=\; \frac{mn}{m+n} \int\limits_{-\infty}^{\infty} \left[F_m(x) - G_n(x)  - \textstyle\int\limits_{-\infty}^{\infty} \left[F_m(r) - G_n(r)\right] dH(r) \right]^2 dH(x) $$
Note that the difference between the distributions is being corrected by subtracting its mean. This correction gives rotation invariant measures of discrepancy between distribution.

In the hypothesis testing application, we reject the null that the compared distributions fit if the suitable test statistic exceeds some critical value. The above tests are \emph{consistent}, meaning that if the underlying distributions differ then the null hypothesis is rejected with probability tending to one as the sample size grows. They are \emph{nonparamteric}, designed for any continuous alternative distributions, with no further assumptions. The statistics are easy to compute, since the empirical distribution functions attain discrete sets of values, so integrals become finite sums. They are conveniently \emph{distribution-free}, that is, if the underlying distributions are equal, then the behaviour of the test statistic does not depend on that particular distribution.

Indeed, it is not difficult to observe that $\omega_{mn}^2$ and $U_{mn}^2$ only rely on the ordering of the given $m+n$ points along the real line, or around the circle, rather than their precise values. It is hence sufficient to summarize the two given samples as one binary word $w \in \{X,Y\}^{m+n}$, encoding which sample each data point comes from, in order of occurrence along the real line or circle. For example, if $X_2 < Y_1 < Y_3 < X_1 < Y_2$ then $w=XYYXY$. The statistic $U_{mn}^2$ is appropriately invariant to rotations of this word. Under the null hypothesis that all $X_i \sim Y_j$, each word with $m$ copies of~$X$ and $n$ copies of~$Y$ is equally likely. 

Persson \cite{persson1979new} showed that several two-sample test statistics such as the above ones can be elegantly expressed in terms of subword counts in the random word~$w$. Using the above notation for the number of occurrences, Persson's formula for Watson's two-sample statistic is as follows.
$$ U_{mn}^2 \;=\; \frac{mn(mn+2)/12 \;-\; \left[\#XYXY(w)+\#YXYX(w)\right]}{mn(m+n)}  $$

Watson~\cite{watson1961goodness, watson1962goodness} showed that, under the null hypotheses, both test statistics converge in law to the same limit: $U_n^2 \to U^2$ as $n \to \infty$, and $U_{mn}^2 \to U^2$ as $m,n \to \infty$ assuming $m/n \to \lambda > 0$.
This asymptotic distribution is a sum of squares of standard normal random variables:
$$ U^2 \;\sim\; \sum\limits_{k=1}^{\infty}\frac{Z_k^2+\tilde{Z}_k^2}{4\pi^2k^2} \;\;\;\; \text{ for independent } Z_k \sim \tilde{Z}_k \sim \mathcal{N}(0,1) $$
See Watson's papers for further details on the distribution of~$U^2$, and its various representations: the probability density function, the cumulative distribution, the moment generating function, an integral of a squared Gaussian process, and a relation to the Kolmogorov--Smirnov test statistic. 

Janson~\cite{janson1984asymptotic} showed that the limit $U_{mn}^2 \to U^2$ actually holds without any restrictions on the relation between $m$ and~$n$, as long as they both tend to~$\infty$.

\medskip

\subsection*{Proof of Theorem~\ref{stcore}}
The theorem is proven in two steps. First we use the size formula to relate the size distribution of $(s,t)$-cores with subword statistics. Then we apply the results of Watson, Persson, and Janson stated above to obtain the distribution explicitly.  

\medskip

\emph{Step 1}:
Anderson showed that $(s,t)$-core partitions are in bijection with $s/t$-ballot words, and in Proposition~\ref{stst} we showed that a certain statistic of the ballot word gives the size of the corresponding core partition. This implies the following equality of distributions between the size $X_{st}$ of a uniformly random partition and the statistic for random words:
$$ X_{st} \;\sim\; \frac{(s^2-1)(t^2-1)}{24} \;-\; \frac{\#STST(w) + \#TSTS(w)}{2} $$
Here the random word $w$ follows the uniform distribution over the set of all the $\tfrac{1}{s+t}\tbinom{s+t}{s}$ words that satisfy the $s/t$-ballot condition. 

We observe that the distribution of the subword count $\#STST+\#TSTS$ over $s/t$-ballot words is the same as its distribution over all the $\tbinom{s+t}{s}$ words with $\#S=s$ and~$\#T=t$. Indeed, this statistic is invariant under cyclic rotation of words, and since $s$ and~$t$ are coprime every orbit of this $\mathbb{Z}_{s+t}$ action has $s+t$ different words, exactly one of which is $s/t$-ballot. To see that, consider the path as on the right hand side of Figure~\ref{bijection}, and note that any nontrivial rotation takes its unique minimum from zero to a negative value, violating the $s/t$ ballot condition. 

We remark that this rotation argument is a variant of the classical \emph{Cycle Lemma} \cite{dvoretzky1947problem,spitzer1956combinatorial}. It provides a short proof of the enumeration of $s/t$ ballot words by the rational Catalan number~$\frac{1}{s+t} \tbinom{s+t}{s}$.

In conclusion, in the equivalent distribution for $X_{st}$ stated above, the random word $w$ may alternatively be taken as uniform among all the words with $\#S(w)=s$ and~$\#T(w)=t$.

\medskip

\emph{Step 2}.
Suppose that the uniformly random word $w$ originates in the random ordering along the real line of $s+t$ independent random variables $S_1,\dots,S_s,T_1,\dots,T_t$ all following the same continuous distribution. By Persson's formula, their two-sample Watson's $U_{st}^2$ satisfies
$$ \frac{st(s+t)}{2} U_{st}^2 \;=\; \frac{st(st+2)}{24} \;-\; \frac{\#STST(w)+\#TSTS(w)}{2} $$
In comparison to the above distribution of~$X_{st}$, the right hand side only differs by
$$ \frac{st(st+2)}{24} \;-\; \frac{(s^2-1)(t^2-1)}{24} \;=\; \frac{(s+t)^2-1}{24} \;=\; o\left(\frac{st(s+t)}{2}\right) $$
as $\min(s,t) \to  \infty$. 

Therefore, the normalized partition size $X_{st}/\tfrac{1}{2}st(s+t)$ converges to the same limit law as the null distribution of~$U_{st}^2$. This asymptotic distribution is given by~$U^2$ based on the above results of Watson and Janson.
\qed

\bigskip

\section*{Acknowledgements}

I would like to thank Doron Zeilberger for his hospitality at Rutgers, and for introducing me to this problem.

\smallskip

I would like to thank Tsviqa Lakrec and Ran Tessler for valuable discussions and feedback.

\smallskip

I would like to thank the two anonymous reviewers for useful comments and suggestions that helped improve the exposition.

\smallskip

I would like to thank the Lloyds Register Foundation / Alan Turing Institute programme on Data-Centric Engineering for their support.

{
\bibliographystyle{alpha}
\bibliography{main}

\begin{thebibliography}{WWY18}

\bibitem[Agg15a]{aggarwal2015armstrong}
Amol Aggarwal.
\newblock Armstrong’s conjecture for $(k,mk+1)$-core partitions.
\newblock {\em European Journal of Combinatorics}, 47:54--67, 2015.

\bibitem[Agg15b]{aggarwal2015does}
Amol Aggarwal.
\newblock When does the set of $(a,b,c)$-core partitions have a unique maximal
  element?
\newblock {\em The Electronic Journal of Combinatorics}, 22(2):P2.31,1--10,
  2015.

\bibitem[AHJ14]{armstrong2014results}
Drew Armstrong, Christopher~RH Hanusa, and Brant~C Jones.
\newblock Results and conjectures on simultaneous core partitions.
\newblock {\em European Journal of Combinatorics}, 41:205--220, 2014.

\bibitem[Amd15]{amdeberhan2015theorems}
Tewodros Amdeberhan.
\newblock Theorems, problems and conjectures.
\newblock {\em Available at arXiv:1207.4045}, 2015.

\bibitem[And02]{anderson2002partitions}
Jaclyn Anderson.
\newblock Partitions which are simultaneously $t_1$-and $t_2$-core.
\newblock {\em Discrete Mathematics}, 248(1-3):237--243, 2002.

\bibitem[AS15]{amdeberhan2015multi}
Tewodros Amdeberhan and Emily {Sergel Leven}.
\newblock Multi-cores, posets, and lattice paths.
\newblock {\em Advances in Applied Mathematics}, 71:1--13, 2015.

\bibitem[BNY18]{baek2018bijective}
Jineon Baek, Hayan Nam, and Myungjun Yu.
\newblock A bijective proof of {A}mdeberhan’s conjecture on the number of
  $(s,s+2)$-core partitions with distinct parts.
\newblock {\em Discrete Mathematics}, 341(5):1294--1300, 2018.

\bibitem[BNY19]{baek2019johnson}
Jineon Baek, Hayan Nam, and Myungjun Yu.
\newblock Johnson’s bijections and their application to counting simultaneous
  core partitions.
\newblock {\em European Journal of Combinatorics}, 75:43--54, 2019.

\bibitem[CHW16]{chen2016average}
William Chen, Harry Huang, and Larry Wang.
\newblock Average size of a self-conjugate $(s,t)$-core partition.
\newblock {\em Proceedings of the American Mathematical Society},
  144(4):1391--1399, 2016.

\bibitem[Cra28]{cramer1928composition}
Harald Cram\'er.
\newblock On the composition of elementary errors.
\newblock {\em Scandinavian Actuarial Journal}, 1928(1):141--180, 1928.

\bibitem[DM47]{dvoretzky1947problem}
Aryeh Dvoretzky and Theodore Motzkin.
\newblock A problem of arrangements.
\newblock {\em Duke Mathematical Journal}, 14(2):305--313, 1947.

\bibitem[ELT20]{even2020spectral}
Chaim Even{-Zohar}, Tsviqa Lakrec, and Ran~J Tessler.
\newblock Spectral analysis of word statistics.
\newblock {\em arXiv preprint arXiv:2012.00742}, 2020.

\bibitem[EZ15]{ekhad2015explicit}
Shalosh~B Ekhad and Doron Zeilberger.
\newblock Explicit expressions for the variance and higher moments of the size
  of a simultaneous core partition and its limiting distribution.
\newblock {\em arXiv preprint arXiv:1508.07637}, 2015.

\bibitem[FMS09]{ford2009self}
Ben Ford, Ho{\`a}ng Mai, and Lawrence Sze.
\newblock Self-conjugate simultaneous $p$ and $q$ core partitions and blocks of
  {$A_n$}.
\newblock {\em Journal of Number Theory}, 129(4):858--865, 2009.

\bibitem[GM14]{gorsky2014compactified}
Evgeny Gorsky and Mikhail Mazin.
\newblock Compactified {J}acobians and $(q,t)$-{C}atalan numbers, {II}.
\newblock {\em Journal of Algebraic Combinatorics}, 39(1):153--186, 2014.

\bibitem[GMV14]{gorsky2014affine}
Eugene Gorsky, Mikhail Mazin, and Monica Vazirani.
\newblock Affine permutations and rational slope parking functions.
\newblock {\em arXiv preprint arXiv:1403.0303}, 2014.

\bibitem[Jan84]{janson1984asymptotic}
Svante Janson.
\newblock The asymptotic distributions of incomplete {U}-statistics.
\newblock {\em Zeitschrift f{\"u}r Wahrscheinlichkeitstheorie und Verwandte
  Gebiete}, 66(4):495--505, 1984.

\bibitem[JK81]{james1981representation}
Gordon James and Adalbert Kerber.
\newblock The representation theory of the symmetric group.
\newblock {\em Encyclopedia Math. Appl}, 1981.

\bibitem[Joh18a]{johnson2018lattice}
Paul Johnson.
\newblock Lattice points and simultaneous core partitions.
\newblock {\em The Electronic Journal of Combinatorics}, 25(3):3--47, 2018.

\bibitem[Joh18b]{johnson2018simultaneous}
Paul Johnson.
\newblock Simultaneous cores with restrictions and a question of {Z}aleski and
  {Z}eilberger.
\newblock {\em arXiv preprint arXiv:1802.09621}, 2018.

\bibitem[KST18]{komlos2018asymptotic}
J{\'a}nos Koml{\'o}s, Emily Sergel, and G{\'a}bor Tusn{\'a}dy.
\newblock The asymptotic normality of $(s,s+1)$-cores with distinct parts.
\newblock {\em arXiv preprint arXiv:1809.00412}, 2018.

\bibitem[Leh51]{lehmann1951consistency}
Eric~L Lehmann.
\newblock Consistency and unbiasedness of certain nonparametric tests.
\newblock {\em The annals of mathematical statistics}, pages 165--179, 1951.

\bibitem[Nak40]{nakayama1940some}
Tadasi Nakayama.
\newblock On some modular properties of irreducible representations of a
  symmetric group, {I}-{II}.
\newblock In {\em Japanese journal of mathematics: transactions and abstracts},
  volume~17, pages 165--184, 411--423. The Mathematical Society of Japan, 1940.

\bibitem[NS17]{nath2017abaci}
Rishi Nath and James~A Sellers.
\newblock Abaci structures of $(s,ms\pm1)$-core partitions.
\newblock {\em The Electronic Journal of Combinatorics}, 24(1):P1.5,1--20,
  2017.

\bibitem[OS07]{olsson2007block}
J{\o}rn~B Olsson and Dennis Stanton.
\newblock Block inclusions and cores of partitions.
\newblock {\em Aequationes mathematicae}, 74(1-2):90--110, 2007.

\bibitem[Par18]{paramonov2018cores}
Kirill Paramonov.
\newblock Cores with distinct parts and bigraded {F}ibonacci numbers.
\newblock {\em Discrete Mathematics}, 341(4):875--888, 2018.

\bibitem[Per79]{persson1979new}
Tore Persson.
\newblock A new way to obtain {W}atson's {$U^2$}.
\newblock {\em Scandinavian Journal of Statistics}, pages 119--122, 1979.

\bibitem[Ros52]{rosenblatt1952limit}
Murray Rosenblatt.
\newblock Limit theorems associated with variants of the von {M}ises statistic.
\newblock {\em The Annals of Mathematical Statistics}, pages 617--623, 1952.

\bibitem[Spi56]{spitzer1956combinatorial}
Frank Spitzer.
\newblock A combinatorial lemma and its application to probability theory.
\newblock {\em Transactions of the American Mathematical Society},
  82(2):323--339, 1956.

\bibitem[Str16]{straub2016core}
Armin Straub.
\newblock Core partitions into distinct parts and an analog of {E}uler’s
  theorem.
\newblock {\em European Journal of Combinatorics}, 57:40--49, 2016.

\bibitem[SZ15]{stanley2015catalan}
Richard~P Stanley and Fabrizio Zanello.
\newblock The {C}atalan case of {A}rmstrong's conjecture on simultaneous core
  partitions.
\newblock {\em SIAM Journal on Discrete Mathematics}, 29(1):658--666, 2015.

\bibitem[Thi16]{thiel2016anderson}
Marko Thiel.
\newblock From {A}nderson to zeta.
\newblock {\em Advances in Applied Mathematics}, 81:156--201, 2016.

\bibitem[Tri09]{tripathi2009largest}
Amitabha Tripathi.
\newblock On the largest size of a partition that is both $s$-core and
  $t$-core.
\newblock {\em Journal of Number Theory}, 129(7):1805--1811, 2009.

\bibitem[TW17]{thiel2017strange}
Marko Thiel and Nathan Williams.
\newblock Strange expectations and simultaneous cores.
\newblock {\em Journal of Algebraic Combinatorics}, 46(1):219--261, 2017.

\bibitem[vM31]{mises1931wahrscheinlichkeitsrechnung}
Richard von Mises.
\newblock {\em Wahrscheinlichkeitsrechnung und {I}hre {A}nwendung in der
  {S}tatistik und {T}heoretischen {P}hysik}.
\newblock Leipzig and Wien, Franz Deuticke, 1931.

\bibitem[Wan16]{wang2016simultaneous}
Victor~Y Wang.
\newblock Simultaneous core partitions: Parameterizations and sums.
\newblock {\em The Electronic Journal of Combinatorics}, 23(1):1--34, 2016.

\bibitem[Wat61]{watson1961goodness}
George~S Watson.
\newblock Goodness-of-fit tests on a circle.
\newblock {\em Biometrika}, 48(1/2):109--114, 1961.

\bibitem[Wat62]{watson1962goodness}
George~S Watson.
\newblock Goodness-of-fit tests on a circle. {II}.
\newblock {\em Biometrika}, 49(1/2):57--63, 1962.

\bibitem[WWY18]{wang2018moments}
Joseph~LP Wang, Larry~XW Wang, and Jane~YX Yang.
\newblock Moments about the mean of the size of a self-conjugate $(s, t)$-core
  partition.
\newblock {\em Discrete Mathematics}, 341(11):3029--3043, 2018.

\bibitem[Xio16]{xiong2016largest}
Huan Xiong.
\newblock On the largest size of $(t,t+1,...,t+p)$-core partitions.
\newblock {\em Discrete Mathematics}, 339(1):308--317, 2016.

\bibitem[Xio18a]{xiong2018core}
Huan Xiong.
\newblock Core partitions with distinct parts.
\newblock {\em The Electronic Journal of Combinatorics}, 25(1):P1.57,1--10,
  2018.

\bibitem[Xio18b]{xiong2018largest}
Huan Xiong.
\newblock On the largest sizes of certain simultaneous core partitions with
  distinct parts.
\newblock {\em European Journal of Combinatorics}, 71:33--42, 2018.

\bibitem[XZ19]{xiong2019polynomiality}
Huan Xiong and Wenston~JT Zang.
\newblock On the polynomiality and asymptotics of moments of sizes for random
  $(n,dn\pm1)$-core partitions with distinct parts.
\newblock {\em Science China Mathematics}, pages 1--18, 2019.

\bibitem[YQJZ17]{yan20172k}
Sherry~HF Yan, Guizhi Qin, Zemin Jin, and Robin~DP Zhou.
\newblock On $(2k+1,2k+3)$-core partitions with distinct parts.
\newblock {\em Discrete Mathematics}, 340(6):1191--1202, 2017.

\bibitem[YZZ15]{yang2015enumeration}
Jane~YX Yang, Michael~XX Zhong, and Robin~DP Zhou.
\newblock On the enumeration of $(s, s+ 1, s+ 2)$-core partitions.
\newblock {\em European Journal of Combinatorics}, 49:203--217, 2015.

\bibitem[Zal17]{zaleski2017explicit-ss1}
Anthony Zaleski.
\newblock Explicit expressions for the moments of the size of an $(s,s+1)$-core
  partition with distinct parts.
\newblock {\em Advances in Applied Mathematics}, 84:1--7, 2017.

\bibitem[Zal19]{zaleski2017explicit-d}
Anthony Zaleski.
\newblock Explicit expressions for the moments of the size of an
  {$(n,dn-1)$}-core partition with distinct parts.
\newblock {\em Integers}, 19:2, 2019.

\bibitem[ZZ17a]{zaleski2017explicit-13}
Anthony Zaleski and Doron Zeilberger.
\newblock Explicit expressions for the expectation, variance and higher moments
  of the size of a $(2n+1,2n+3)$-core partition with distinct parts.
\newblock {\em Journal of Difference Equations and Applications},
  23(7):1241--1254, 2017.

\bibitem[ZZ17b]{zaleski2017intriguing}
Anthony Zaleski and Doron Zeilberger.
\newblock On the intriguing problem of counting $(n+1,n+2)$-core partitions
  into odd parts.
\newblock {\em arXiv preprint arXiv:1712.10072}, 2017.

\end{thebibliography}
}
\end{document}